\documentclass[12pt]{amsart}
\usepackage{amssymb}
\usepackage{amsmath}
\usepackage{amssymb}
\usepackage{amsthm}
\usepackage{enumerate}
\usepackage[utf8]{inputenc}
\usepackage[all,arc,curve,color,frame]{xy}
\usepackage[utf8]{inputenc}

\DeclareMathOperator{\End}{End}

\DeclareMathOperator{\tr}{tr}
\DeclareMathOperator{\ad}{ad}

\newcommand{\A}{{\mathcal A}}
\newcommand{\B}{{\mathcal B}}

\newcommand{\E}{{\mathcal E}}
\newcommand{\F}{{\mathcal F}}

\newcommand{\I}{{\mathcal I}}

\newcommand{\K}{{\mathcal K}}

\newcommand{\Md}{{M_2}}

\renewcommand{\P}{{\mathcal P}}

\newcommand{\C}{\ensuremath{\mathbb{C}}}
\newcommand{\R}{\ensuremath{\mathbb{R}}}

\newcommand{\p}{\partial}

\newcommand{\unit}{{\bf{1}}}

\newcommand{\U}{{\mathcal U}}
\newcommand{\V}{{\mathcal V}}

\newcommand{\x}{{x_0}}

\newcommand{\Z}{\ensuremath{\mathbb{Z}}}
\newcommand{\ba}{\begin{eqnarray*}}
\newcommand{\ea}{\end{eqnarray*}}

\newtheorem{definition}{Definition}[section]
\newtheorem{lemma}{Lemma}[section]
\newtheorem{proposition}{Proposition}[section]
\newtheorem{theorem}{Theorem}[section]
\newtheorem{corollary}{Corollary}[section]
\begin{document}

\title[Lagrangian fields]
{Lagrangian fields, Calabi functions, and local symplectic groupoids}
\author[Alexander Karabegov]{Alexander Karabegov}
\address[Alexander Karabegov]{Department of Mathematics, Abilene
Christian University, ACU Box 28012, Abilene, TX 79699-8012}
\email{axk02d@acu.edu}


\begin{abstract}
A Lagrangian field on a symplectic manifold $M$ is a family $\Lambda=\{\Lambda_x|x \in M\}$ of pointed Lagrangian submanifolds of $M$. This notion is a generalization of a real Lagrangian polarization for which each $\Lambda_x$ is the leaf containing $x$. Two Lagrangian fields $\Lambda$ and $\tilde \Lambda$ are called transversal if $\Lambda_x$ intersects $\tilde\Lambda_x$  transversally at $x$ for every $x$. Two transversal Lagrangian fields determine an almost para-K\"ahler structure on $M$.
We construct a local symplectic groupoid on a neighborhood of the zero section of $T^\ast M$ from two transversal Lagrangian fields on $M$. The Lagrangian manifold of $n$-cycles of this groupoid in $(T^\ast M)^n$ has a generating function whose germ around the diagonal of $M^n$ is given by the $n$-point cyclic Calabi function of a closed (1,1)-form on a neighborhood of the diagonal of $M^2$ obtained from the symplectic form on~$M$.
\end{abstract}

\maketitle

\section{Introduction}

A symplectic groupoid over a Poisson manifold $M$ is a geometric object which is a heuristic semiclassical counterpart of an algebra of quantum observables on $M$. Symplectic groupoids were first introduced by Karasev \cite{Kar}, Weinstein \cite{W}, and Zakrzewski \cite{Z}. Symplectic groupoids can be treated as a starting point for quantization (see \cite{CDF} and \cite{Haw}). 

In this paper we construct local symplectic groupoids over almost para-K\"ahler manifolds. These groupoids are based on an additional geometric structure, a Lagrangian field, which is a generalization of a Lagrangian polarization. A Lagrangian field on a symplectic manifold $M$ is a family $\Lambda=\{\Lambda_x|x \in M\}$ of pointed Lagrangian submanifolds of $M$ (with $x \in \Lambda_x$). It induces a Lagrangian distribution $\{ T_x\Lambda_x, x \in M\}$ on $M$ which is not necessarily integrable.

Two Lagrangian fields $\Lambda$ and $\tilde \Lambda$ are called transversal if $\Lambda_x$ intersects $\tilde\Lambda_x$  transversally at $x$ for every $x \in M$. The corresponding transversal Lagrangian distributions determine an almost para-K\"ahler structure on~$M$. 

We construct a local symplectic groupoid on a neighborhood of the zero section of $T^\ast M$ from two transversal Lagrangian fields on $M$. Such groupoids have a rich geometric structure. We expect that they can be used for quantization of almost para-K\"ahler manifolds.

\section{General facts on groupoids}

In this section we give basic facts on groupoids and fix various conventions that will be used in the rest of the paper.  

A groupoid $\Sigma \rightrightarrows M$ is a small category with the set of objects $M$ and the set of morphisms $\Sigma$, where all morphisms are invertible. The injective unit mapping $\epsilon:M \to\Sigma$ maps an object $x$ to the identity morphism $\mathrm{id}_x : x \to x$. The set $\epsilon(M)$ of the identity morphisms in $\Sigma$ is called the set of units of the groupoid. The source and target mappings $s,t:\Sigma\to M$ map a morphism $\alpha: x \to y$ to $s(\alpha)=x$ and $t(\alpha)=y$. The involutive inverse mapping $i:\Sigma\to\Sigma$ maps $\alpha \in \Sigma$ to $i(\alpha)=\alpha^{-1}$. We will write the composition of $\alpha:y\to z$ and $\beta: x \to y$ as $\alpha\beta:x\to z$. Two morphisms $\alpha$ and $\beta$ are thus composable if $s(\alpha)=t(\beta)$. An element $(\alpha_1, \ldots,\alpha_n)\in \Sigma^n$ is an $n$-cycle in $\Sigma$ if the composition of the elements $\alpha_1, \ldots,\alpha_n\in \Sigma$ is defined and equals an identity morphism,
\[
\alpha_1\alpha_2 \ldots\alpha_n=\mathrm{id}_x,
\]
where $x= s(\alpha_n)= t(\alpha_1)$. We denote the set of $n$-cycles in $\Sigma$ by $\Sigma_n$. So, the graph of $i$ is $\Sigma_2$.

The set of objects of every groupoid in this paper is denoted by $M$. We denote the diagonal of the set $M^n$ by $M_n$, which should not lead to confusion \footnote{This notation agrees with that for the set of $n$-cycles if $M$ is treated as the identity groupoid where every morphism is an identity.}.

{\it Example} The pair groupoid of a set $M$ has the set of objects $M$ and the set of morphisms $M^2$ with $\epsilon(x)=(x,x), i(x,y)=(y,x), s(x,y)=y$, and $t(x,y)=x$, so that the unique morphism $\alpha: y \to x$ is $\alpha=(x,y)$. The set of units is the diagonal $\Md \subset M^2$. The composition of $(x,y)$ and $(y,z)$ is $(x,y)(y,z)=(x,z)$. An $n$-cycle is
\[
\{\alpha_1,\alpha_2, \ldots, \alpha_n\}=\{(x_1,x_2), (x_2,x_3), \ldots, (x_n,x_1)\},
\]
where $\alpha_i: x_{i+1} \to x_i$ for $1 \leq i \leq n-1$ and $\alpha_n:x_1\to x_n$.

A groupoid $\Sigma \rightrightarrows M$ is a Lie groupoid if $M$ and $\Sigma$ are manifolds, all structure mappings are smooth, the unit mapping is an embedding, and the source and target mappings are surjective submersions. A local Lie groupoid has the same properties as a Lie groupoid, but in order for two elements to be composable they should also be close to each other (see \cite{CDW} and \cite{AC}). A neighborhood of the unit space of a Lie groupoid is a local Lie groupoid.

A symplectic groupoid $\Sigma \rightrightarrows M$ is a Lie groupoid where $\Sigma$ is a symplectic manifold, the space of units $\epsilon(M)$ is Lagrangian, $i$ is an anti-symplectomorphism, and the set $\Sigma_3$ of 3-cycles is a Lagrangian submanifold of $\Sigma^3$ \footnote{In the conventional definition of a symplectic groupoid an equivalent condition is used: the product space $\{(\alpha, \beta, \gamma) | s(\alpha)=t(\beta), \gamma=\alpha\beta\}$ is a Lagrangian submanifold of $\Sigma \times \Sigma \times \bar\Sigma$, where $\bar \Sigma$ is a copy of $\Sigma$ with the opposite symplectic structure.}. It can be shown that then the space $\Sigma_n$ of $n$-cycles is a Lagrangian submanifold of $\Sigma^n$ for all $n\geq 2$. There exists a unique Poisson structure on $M$ such that $s$ is a Poisson morphism and $t$ is an anti-Poisson morphism. 

If $M = (M,\omega)$ is a symplectic manifold and $\bar M=(M, -\omega)$ is a copy of $M$ with the opposite symplectic structure, then the pair groupoid $\bar M \times M \rightrightarrows  M$ with the product symplectic structure on $\bar M \times M$ is a symplectic groupoid over $M$. Its diagonal unit space $\Md$ is Lagrangian. 

If $M$ is compact, by Weinstein's neighborhood theorem there exists a neighborhood $V$ of  the zero section $Z_{T^\ast M}$ in $T^\ast M$ and a mapping $f: V \to \bar M \times M$ which identifies $Z_{T^\ast M}$ with the diagonal $\Md$ of $\bar M \times M$ as copies of $M$ (i.e., if $z_x$ is the unique element in $Z_{T^\ast M} \cap T^\ast_x M$, then $f(z_x)=(x,x)$ for all $x \in M$) and is a symplectomorphism of $V$ onto $f(V)$. Then $V \rightrightarrows M$ is a local symplectic groupoid over $M$ with $s = \pi_2 \circ f$ and $t=\pi_1 \circ f$, where $\pi_1$ and $\pi_2$ are the projections of $\bar M \times M$ onto the respective factors.

Given a symplectic manifold $M$ equipped with two transversal Lagrangian fields $\Lambda$ and $\tilde\Lambda$, we construct in this paper a local symplectic groupoid $V \rightrightarrows M$, where $V$ is a neighborhood of $Z_{T^\ast M}$ in $T^\ast M$ (the compactness of  $M$ is not required). For this groupoid, $s(T^\ast_x M\cap V) \subset \Lambda_x$ and $t(T^\ast_x M \cap V) \subset \tilde\Lambda_x$ for all $x \in M$. We also show that the Lagrangian manifold of $n$-cycles of this groupoid has a generating function whose germ around the diagonal $M_n$ of $M^n$ is given by the $n$-point cyclic Calabi function of a closed (1,1)-form on a neighborhood of the diagonal of $M^2$ obtained from the symplectic form on~$M$.

\medskip

{\bf Acknowledgments} I am very grateful to Th.Th.Voronov for an important discussion on Lagrangian fields and to Alan Weinstein for valuable suggestions.

\section{A local symplectic groupoid generated by two transversal Lagrangian fields}\label{S:lf}

Below we give a technical definition of a Lagrangian field on a symplectic manifold $(M,\omega)$ of dimension $2m$. 

\begin{definition}
A Lagrangian field $\Lambda  = (E, \lambda)$ on $M$ is a vector bundle $\pi: E \to M$ of rank $m$ equipped with a surjective submersion $\lambda: E \to M$ whose restriction to each fiber $E_x, x \in M$, is an embedding onto a Lagrangian submanifold $\Lambda_x := \lambda(E_x)$ of $M$ and such that $\lambda=\pi$ on the zero section $Z_E$ of $E$.
\end{definition}

If $z_x$ is the unique element of $Z_E \cap E_x$, we have $\lambda(z_x) = \pi(z_x)=x$,  whence $x \in \Lambda_x$ for all $x \in M$. 

\begin{lemma}\label{L:xv}
There exists a unique horizontal 1-form $\phi$ on the total space of $E$ which vanishes on $Z_E$ and satisfies the equation
\begin{equation}\label{E:dvarphi}
\lambda^\ast \omega - \pi^\ast\omega = d\phi.
\end{equation}
\end{lemma}
\begin{proof}
Let $F:[0,1] \times E \to M$ be the homotopy $F(t,z) = \lambda(t \cdot z)$ between the mappings $\pi$ and $\lambda$, where $z \mapsto t \cdot z$ is the scalar multiplication by $t$ on the fibers of $E$. Denote by $i_t: E \to [0,1]\times E$ the inclusion mapping $i_t(z) = (t, z)$. We will show that the form
\[
\phi := \int_0^1 i^\ast_t \iota_{\frac{\p}{\p t}} F^\ast \omega\, dt,
\]
where $\iota_{\frac{\p}{\p t}}$ is an insertion operator, satisfies the conditions of the lemma. Equality (\ref{E:dvarphi}) holds by the homotopy formula. For any $x\in M$, the mapping $F$ maps $[0,1] \times E_x$ onto the Lagrangian manifold $\Lambda_x$, whence
\[
   F^\ast \omega|_{[0,1] \times E_x} = 0.
\]
If $z \in Z_E$, then $F$ maps $[0,1]\times \{z\}$ to the point $\pi(z)$. It implies that the form $i^\ast_t \iota_{\frac{\p}{\p t}} F^\ast \omega$ on $E$ is horizontal and its restriction to $Z_E$ vanishes. Hence, the same holds true for the 1-form $\phi$. 

Let $\psi$ be a closed horizontal 1-form on $E$ vanishing on $Z_E$. If $\eta$ is a vertical vector field on $E$, then $\iota_\eta \psi= 0$. Therefore, $\mathcal{L}_\eta \psi=0$, where $\mathcal{L}_\eta$ is the Lie derivative with respect to $\eta$. It follows that $\psi$ is the pullback of a 1-form on $M$ by $\pi$. Since the restriction of $\psi$ to $Z_E$ vanishes, we see that $\psi=0$, which implies the uniqueness of $\phi$. Thus, the 1-form $\phi$ satisfies all conditions of the lemma. 
\end{proof}

Let $\Lambda=(E,\lambda)$ and $\tilde\Lambda=(\tilde E, \tilde\lambda)$ be two Lagrangian fields on $M$. We denote the bundle projection for $\tilde E$ by $\tilde\pi$,  the zero section by $Z_{\tilde E}$, the unique element of $Z_{\tilde E} \cap \tilde E_x$ by $\tilde z_x$, and set $\tilde \Lambda_x := \tilde\lambda(\tilde E_x)$ for each $x\in M$. According to Lemma \ref{L:xv}, there exists a unique horizontal 1-form $\tilde\phi$ on $\tilde E$ vanishing on $Z_{\tilde E}$ and satisfying
\begin{equation}\label{E:dtvarphi}
\tilde\lambda^\ast \omega - \tilde\pi^\ast\omega = d \tilde\phi.
\end{equation}
Consider the vector bundle
\[
\hat E := E \oplus \tilde E = \{(z,\tilde z) \in E \times \tilde E| \pi(z)=\tilde\pi(\tilde z)\}.
\]
The bundle projection $\hat \pi: \hat E \to M$ is such that $\hat\pi (z, \tilde z) = \pi(z)=\tilde\pi(\tilde z)$. We introduce projections $\tau:\hat E \to E$ and $\tilde\tau:\hat E \to \tilde E$ given by $\tau(z,\tilde z)=z$ and $\tilde\tau(z,\tilde z)=\tilde z$. Then $\hat \pi = \pi \tau = \tilde\pi \tilde\tau$. We get the following pullbacks of (\ref{E:dvarphi}) and (\ref{E:dtvarphi}) by $\tau$ and $\tilde\tau$, respectively:
\[
\tau^\ast \lambda^\ast \omega - \hat\pi^\ast \omega = d \tau^\ast \phi \mbox{ and } \tilde \tau^\ast \tilde \lambda^\ast \omega - \hat\pi^\ast \omega = d \tilde \tau^\ast \tilde\phi.
\]
Subtracting these equations, we get that
\begin{equation}\label{E:ident}
\tau^\ast \lambda^\ast \omega - \tilde \tau^\ast \tilde \lambda^\ast \omega = -d(\tilde \tau^\ast \tilde\phi - \tau^\ast \phi).
\end{equation}
The 1-form $\tilde \tau^\ast \tilde\phi - \tau^\ast \phi$ is horizontal on $\hat E$ and vanishes on the zero section $Z_{\hat E}$. It determines a fiberwise nonlinear mapping
\[
\chi: \hat E \to T^\ast M
\]
which maps a point $(z, \tilde z) \in \hat E_x$ to the unique element $\chi(z, \tilde z) \in T^\ast_x M$ whose pullback by $\hat\pi$ is
\[
(d\hat\pi_{(z,\tilde z)})^\ast  \chi(z, \tilde z) = (\tilde \tau^\ast \tilde\phi - \tau^\ast \phi)(z,\tilde z) \in T^\ast_{(z,\tilde z)} \hat E.
\]
The mapping $\chi$ identifies the zero sections $Z_{\hat E}$ and $Z_{T^\ast M}$ as copies of~$M$.

Let $\theta$ be the canonical 1-form on $T^\ast M$. If $\{x^i\}$ are local coordinates on $M$ and $\{\xi_i\}$ are the dual fiber coordinates on $T^\ast M$, then $\theta= \xi_i dx^i$, where we use summation over repeated lower and upper indices. We have
\[
   \chi^\ast \theta = \tilde \tau^\ast \tilde\phi - \tau^\ast \phi.
\]
Thus, the left-hand side of (\ref{E:ident}) is equal to the pullback of the canonical symplectic form $-d\theta$ on $T^\ast M$ by the mapping $\chi$. \begin{definition}
Lagrangian fields $\Lambda$ and $\tilde \Lambda$ on $M$ are called transversal if $\Lambda_x$ transversally intersects $\tilde\Lambda_x$ at $x$ for all $x \in M$.
\end{definition}

\begin{lemma}\label{L:neighb}
Let $\Lambda$ and $\tilde\Lambda$ be two transversal Lagrangian fields on $M$.
\begin{enumerate}[(i)]
\item There exists a neighborhood $W_1$ of $Z_{\hat E}$ in $\hat E$ such that $\chi|_{W_1}$ is a diffeomorphism of $W_1$ onto a
neighborhood of $Z_{T^\ast M}$ in $T^\ast M$.
\item There exists a neighborhood $W_2$ of $Z_{\hat E}$ in $\hat E$ such that the mapping
\[
\tilde \lambda \tilde\tau \times \lambda \tau |_{W_2}: (z, \tilde z) \mapsto (\tilde \lambda(\tilde z), \lambda(z))
\]
is a diffeomorphism of $W_2$ onto a neighborhood of the diagonal $M_2$ in $M^2$.
\end{enumerate}
\end{lemma}
\begin{proof}
Let $\x$ be an arbitrary point in $M$. Since $\lambda$ and $\tilde \lambda$ identify $Z_E$ and $Z_{\tilde E}$, respectively, with $M$, there exists a coordinate chart $(\hat U, \{x^i\})$ on $M$ around $\x$, trivializations $U \times \R^m$ of $E$ and $\tilde E$ over a neighborhood $U$ of $\x$ in $\hat U$, and an open ball $B$ in $\R^m$ centered at the origin such that $\lambda(U \times B) \subset \hat U$ and $\tilde\lambda(U \times B) \subset \hat U$. We denote the coordinates on $\R^m$ by $\{v^\alpha\}$ and identify $U \times \R^m$ with coordinate charts on $E$ and $\tilde E$. Then locally a point in $E$ or $\tilde E$ can be written as $(x,v)$. The transversality of the Lagrangian fields $\Lambda$ and $\tilde \Lambda$ means that the $2m \times 2m$-matrix
\begin{equation}\label{E:Jac3}
\begin{bmatrix}
-\frac{\p \lambda^i}{\p v^\alpha}(x,0) & \frac{\p {\tilde \lambda}^i}{\p v^\beta}(x,0) 
\end{bmatrix}
\end{equation}
is nondegenerate for all $x \in U$. These trivializations of $E$ and $\tilde E$ induce a trivialization $U \times \R^m \times \R^m$ of $\hat E$ over $U$. A point $(z, \tilde z) \in \hat E$ with $z=(x,v)$ and $\tilde z = (x,w)$ will have the coordinates $(x,v,w)$. 

To prove (i), consider the neighborhood $U \times B \times B$ in $\hat E$ and let $\{\xi_i\}$ be the fiber coordinates on $T^\ast U$ dual to $\{x^i\}$. 
If $\phi=\phi_i(x,v) dx^i$ on $U \times \R^m \approx E|_U$ and $\tilde\phi=\tilde\phi_i(x,v) dx^i$ on $U \times \R^m \approx \tilde E|_U$, the mapping $\chi$ will be written in coordinates as
 \[
 x^i = x^i \quad \xi_j ={\tilde\phi}_j(x,w) - \phi_j(x,v).
 \]
The Jacobian of this mapping is
 \[
\begin{bmatrix}
\delta^i_k & 0 & 0\\
\frac{\p \tilde\phi_j}{\p x^k}(x,w) - \frac{\p \phi_j}{\p x^k}(x,v) & -\frac{\p \phi_j}{\p v^\alpha}(x,v) & \frac{\p {\tilde\phi}_j}{\p v^\beta}(x,w)
\end{bmatrix}.
\]
Since $\phi_i(x,0)=0$ and $\tilde\phi_i(x,0)=0$, at the point $(x,0,0)$ the Jacobian is equal to
\begin{equation}\label{E:Jac4}
\begin{bmatrix}
\delta^i_k & 0 & 0\\
0 & -\frac{\p \phi_j}{\p v^\alpha}(x,0) & \frac{\p {\tilde\phi}_j}{\p v^\beta}(x,0)
\end{bmatrix}.
\end{equation}
We will show that the matrix
\begin{equation}\label{E:Jac5}
\begin{bmatrix}
-\frac{\p \phi_j}{\p v^\alpha}(x,0) & \frac{\p {\tilde\phi}_j}{\p v^\beta}(x,0)
\end{bmatrix}
\end{equation}
is nondegenerate. In local coordinates we write
\[
\omega = \frac{1}{2} \omega_{ij}(x) dx^i \wedge dx^j
\]
and get from (\ref{E:dvarphi}) that
\[
  \frac{\p \phi_k}{\p v^\alpha} = \omega_{ij}(\lambda) \frac{\p \lambda^i}{\p v^\alpha} \frac{\p \lambda^j}{\p x^k}.
\]
It follows from the equality $\lambda(x,0)=x$ that
\[
\frac{\p \phi_k}{\p v^\alpha}(x,0) =  \frac{\p \lambda^i}{\p v^\alpha}(x,0) \omega_{ik}(x).
\]
Similarly,
\[
\frac{\p \tilde\phi_k}{\p v^\alpha}(x,0) =  \frac{\p \tilde\lambda^i}{\p v^\alpha}(x,0) \omega_{ik}(x).
\]
Therefore, we have the following matrix product representation of (\ref{E:Jac5}),
\[
\begin{bmatrix}
-\frac{\p \phi_k}{\p v^\alpha}(x,0) & \frac{\p \tilde\phi_k}{\p v^\beta}(x,0)
\end{bmatrix}
= \begin{bmatrix}
- \frac{\p \lambda^i}{\p v^\alpha}(x,0) & \frac{\p \tilde\lambda^i}{\p v^\beta}(x,0)
\end{bmatrix}
\cdot
\begin{bmatrix}
\omega_{ik}(x)
\end{bmatrix}.
\]
Since the matrix (\ref{E:Jac3}) is nondegenerate, we see that (\ref{E:Jac5}) is also nondegenerate, which implies that the Jacobian (\ref{E:Jac4}) is nondegenerate as well.
Now part (i) of the lemma follows from the fact that the mapping $\chi$ identifies $Z_{\hat E}$ with $Z_{T^\ast M}$ as copies of $M$ and from the generalized inverse function theorem.

To prove part (ii), consider a neighborhood $U \times B \times B$ as in part~(i) and the mapping
\begin{equation}\label{E:triple}
(x,v,w) \mapsto (\tilde\lambda(x,w), \lambda(x,v))
\end{equation}
from $U \times B \times B$ to $M \times M$. It maps $(x,0,0)$ to $(x,x)$. Its Jacobian at $(x,v,w)$ is
\[
\begin{bmatrix}
\frac{\p \tilde\lambda^i}{\p x^k}(x,w) & 0 & \frac{\p \tilde\lambda^i}{\p v^\beta}(x,w)\\
\frac{\p \lambda^j}{\p x^k}(x,v) & \frac{\p \lambda^j}{\p v^\alpha}(x,v) & 0
\end{bmatrix}.
\]
At the point $(x,0,0)$ the Jacobian is equal to
\[
\begin{bmatrix}
\delta^i_k & 0 & \frac{\p \tilde\lambda^i}{\p v^\beta}(x,0)\\
\delta^j_k  & \frac{\p \lambda^j}{\p v^\alpha}(x,0) &0
\end{bmatrix}.
\]
It is nondegenerate because (\ref{E:Jac3}) is nondegenerate. Part (ii) also follows from the generalized inverse function theorem.
\end{proof}

Using the notations from Lemma \ref{L:neighb}, we set
\[
W := W_1\cap W_2 \subset \hat E\mbox{ and }V:= \chi (W) \subset T^\ast M.
\]
There exists a unique mapping $f: V \to M^2$ which is a diffeomorphism of $V$ onto $f(V)=(\tilde \lambda \tilde\tau \times \lambda \tau)(W)$ satisfying
\begin{equation}\label{E:tll}
f \circ \chi |_{W} = \tilde \lambda \tilde\tau \times \lambda \tau |_{W}.
\end{equation}
The neighborhood $V$ contains $Z_{T^\ast M}$ and $f(V)$ contains the diagonal $M_2$ of $M^2$. The mapping $f$ identifies $Z_{T^\ast M}$ with $M_2$ as copies of~$M$. Recall that $\pi_1$ and $\pi_2$ are the projections of $M^2$ onto the respective factors.

\begin{theorem}\label{T:locgr}
The mapping $f$ determines a local symplectic groupoid $V \rightrightarrows M$ with $s= \pi_2 \circ f$ and $t= \pi_1\circ f$ such that
\begin{equation}\label{E:incl}
s(T^\ast_x M \cap V) \subset \Lambda_x\mbox{ and }t(T^\ast_x M \cap V) \subset \tilde\Lambda_x
\end{equation}
for all $x \in M$.
\end{theorem}
\begin{proof}
Eqn (\ref{E:ident}) implies that
\[
f^\ast (\pi^\ast_2 \omega - \pi^\ast_1 \omega)|_V = -d\theta|_V,
\]
which means that $V$ is a local symplectic groupoid over $M$ with $s$ and $t$ as in the statement of the lemma. The inclusions (\ref{E:incl}) hold true because $f$ maps $T^\ast_x M \cap V$ to the set
\[
\{(\tilde\lambda(\tilde z), \lambda(z))| z \in E_x, \tilde z \in \tilde E_x\}
\]
and for $z \in E_x$ and $\tilde z \in \tilde E_x$ one has $\lambda(z) \in \Lambda_x$ and $\tilde\lambda(\tilde z)\in \tilde\Lambda_x$.
\end{proof}

It follows from (\ref{E:tll}) that
\begin{equation}\label{E:sandt}
s \circ \chi |_{W} = \lambda \tau |_{W} \mbox{ and } t \circ \chi |_{W} = \tilde \lambda \tilde \tau  |_{W}.
\end{equation}

\section{The intersection mapping $\gamma$}\label{S:gamma}

In this section we will show that for two transversal Lagrangian fields $\Lambda = (E, \lambda)$ and $\tilde \Lambda=(\tilde E, \tilde\lambda)$ on $M$ there exists a neighborhood $V_\Gamma$ of the diagonal $M_2$ in $M^2$ and a mapping $\gamma:V_\Gamma \to M$ such that $\gamma(x,y) \in \Lambda_x \cap \tilde \Lambda_y$ for all $(x,y)\in V_\Gamma$. We will also show that $\Omega:=\gamma^\ast \omega$ is a closed (1,1)-form with respect to the product structure on $V_\Gamma$.

Consider the surjective submersion $(z, \tilde z) \mapsto (\lambda(z),\tilde\lambda(\tilde z))$ of $E \times \tilde E$ onto $M \times M$. The inverse image of the diagonal $M_2$ of $M^2$ is a $4m$-dimensional submanifold
\[
\Gamma:= \{(z, \tilde z) \in E \times \tilde E | \lambda(z)=\tilde\lambda(\tilde z)\}
\]
of $E \times \tilde E$. Because $\lambda(z) = \pi(z)$ for $z \in Z_E$ and $\tilde\lambda(\tilde z) = \tilde\pi(\tilde z)$ for $\tilde z \in Z_{\tilde E}$, the zero section $Z_{\hat E}= \{(z,\tilde z) \in Z_E \times Z_{\tilde E} | \pi(z)=\tilde\pi(z)\}$ of $\hat E$ is also a submanifold of $\Gamma$.

\begin{lemma}\label{L:gamma}
If $\Lambda$ and $\tilde\Lambda$ are transversal Lagrangian fields on $M$, there exists a neighborhood $W_\Gamma$ of $Z_{\hat E}$ in $\Gamma$ and a neighborhood $V_\Gamma$ of the diagonal $M_2$ in $M^2$ such that the mapping
\[
\pi \times \tilde\pi \big |_{W_\Gamma}: (z, \tilde z) \mapsto (\pi(z), \tilde\pi(\tilde z))
\]
is a diffeomorphism of $W_\Gamma$ onto $V_\Gamma$.
\end{lemma}
\begin{proof}
Let $\x$ be an arbitrary point in $M$.  As in the proof of Lemma \ref{L:neighb}, we fix a coordinate chart $(\hat U, \{x^i\})$ on $M$ around $\x$, a neighborhood $U$ of $\x$ in $\hat U$, trivializations $U \times \R^m$ of $E$ and $\tilde E$ over $U$, and an open ball $B$ in $\R^m$ centered at the origin such that $\lambda(U \times B) \subset \hat U$ and $\tilde\lambda(U \times B) \subset \hat U$. In the calculations below we use local coordinates $(x,v)$ on $U \times \R^m \approx E|_U$ and $(y,w)$ on $U \times \R^m \approx \tilde E|_U$. Consider the function
\[
F(x, v, y, w) := \lambda(x, v) - \tilde\lambda(y, w)
\]
from $U \times B \times U \times B$ to $\R^{2m}$. Then $F(x,0,x,0)=0$, because $\lambda(x,0) =x$ and $\tilde\lambda(x,0)=x$. Since $\Lambda$ and $\tilde\Lambda$ are transversal, the matrix
\[
\begin{bmatrix}
\frac{\p F^i}{\p v^\alpha} (x,0,x,0) & \frac{\p F^i}{\p w^\beta}(x,0,x,0)
\end{bmatrix}
=
\begin{bmatrix}
\frac{\p\lambda^i}{\p v^\alpha}(x,0) & - \frac{\p \tilde\lambda^i}{\p w^\beta}(x,0)
\end{bmatrix}
\]
is nondegenerate for all $x \in U$. By the implicit function theorem, there exists a neighborhood $U_0$ of $\x$ in $U$ and unique mappings $g(x,y)$ and $\tilde g(x,y)$ from $U_0^2$ to $B$ such that $g(\x,\x)=0,\tilde g(\x,\x)=0$, and
\[
F(x, g(x,y),y,\tilde g(x,y))=\lambda(x,g(x,y)) - \tilde\lambda(y, \tilde g(x,y))=0
\]
for all $(x,y)\in U_0^2$. Then the mapping
\[
G(x,y) := (x, g(x,y), y, \tilde g(x,y))
\]
maps $U_0^2$ to~$\Gamma$. In particular, $G(\x,\x)=(\x,0,\x,0)$. It is a diffeomorphism of $U_0^2$ onto $G(U_0^2)$ whose inverse is $\pi \times \tilde\pi \big |_{G(U_0^2)}$. Thus the mapping
\[
\pi \times \tilde\pi \big |_\Gamma : \Gamma \to M^2
\]
is a submersion at all points of $Z_{\hat E}$. It also identifies $Z_{\hat E}$ with $M_2$ as copies of $M$. The lemma follows from the generalized inverse function theorem.
\end{proof}

Lemma \ref{L:gamma} implies that there exists a unique mapping $\gamma: V_\Gamma \to M$ such that
\begin{equation}\label{E:defg}
\gamma(\pi(z), \tilde\pi(\tilde z)) = \lambda(z)=\tilde\lambda(\tilde z)
\end{equation}
for all $(z, \tilde z) \in W_\Gamma$. The main property of this mapping is that
\begin{equation}\label{E:maing}
\gamma(x,y) \in \Lambda_x \cap \tilde \Lambda_y
\end{equation}
for all $(x,y)\in V_\Gamma$. We will show that $\gamma(x,x)=x$ for all $x \in M$. We have
$(z_x,\tilde z_x) \in Z_{\hat E} \subset W_\Gamma$ and
\[
\gamma(x,x) = \gamma(\pi(z_x), \tilde\pi(\tilde z_x)) = \lambda(z_x) = \pi (z_x) = x.
\]
Let $\epsilon: M \to M^2$ be the diagonal inclusion, $\epsilon(x)=(x,x)$ for $x\in M$.  We have shown that
\begin{equation}\label{E:geps}
\gamma \epsilon = \mathrm{id}_M.
\end{equation} 
We introduce a closed  2-form on $V_\Gamma$,
\[
\Omega := \gamma^\ast \omega.
\]
It follows from (\ref{E:geps}) that $\epsilon^\ast \Omega = \omega$ and therefore
\begin{equation}\label{E:omom}
\gamma ^\ast\epsilon^\ast \Omega = \Omega.
\end{equation} 
\begin{lemma}\label{L:oneone}
The form $\Omega$ is a $(1,1)$-form with respect to the product structure on $M^2$.
\end{lemma}
\begin{proof}
The tangent space $T_{(x,y)} V_\Gamma$ at any point $(x,y)\in V_\Gamma$ is the direct sum of the subspaces of vectors of type $(1,0)$ and $(0,1)$ with respect to the product structure on $M^2$,
\[
T_{(x,y)} V_\Gamma = T^{(1,0)}_{(x,y)} V_\Gamma \oplus T^{(0,1)}_{(x,y)} V_\Gamma.
\]
Any tangent vector $u \in T^{(1,0)}_{(x,y)} V_\Gamma$ is the tangent vector at the initial point $(x(0),y)=(x,y)$ of some path $t \mapsto (x(t),y)$ in $V_\Gamma$. It follows from (\ref{E:maing}) that $\gamma$ maps this path to a path in the Lagrangian manifold $\tilde\Lambda_y$ with the initial point $\gamma(x,y)$. Therefore,
\[
d\gamma_{(x,y)} (u) \in T_{\gamma(x,y)} \tilde \Lambda_y.
\]
Similarly, if $u \in T^{(0,1)}_{(x,y)} V_\Gamma$, then
\[
d\gamma_{(x,y)} (u) \in T_{\gamma(x,y)} \Lambda_x.
\]
Thus, if tangent vectors $u$ and $u^\prime$ in $T_{(x,y)} V_\Gamma$ are both of type $(1,0)$ or both of type $(0,1)$, then $\Omega(u,u^\prime)=0$. It implies that $\Omega$ is of type $(1,1)$ with respect to the product structure on $M^2$. 
\end{proof}

\section{Calabi functions}

In \cite{C} Calabi introduced a diastasic function $D_X$ of a real-analytic K\"ahler form $\omega_X$ on a K\"ahler manifold $X$. A local real-analytic potential $\Phi_X(z,\bar z)$ of $\omega_X$ is defined up to a summand which is the real part of an analytic function. If $\Phi(z,\bar w)$ is a complex analytic continuation of $\Phi(z,\bar z)$, the diastatic function is locally defined as follows:
\[
D_X= \Phi_X(z, \bar w) + \Phi_X(w, \bar z) - \Phi_X(z,\bar z) - \Phi_X(w,\bar w).
\]
The function $D_X$ does not depend on the choice of the local potential $\Phi_X$ and is defined on a neighborhood of the diagonal of $X^2$. One can similarly define multivariate functions corresponding to a closed (1,1)-form on a product space. 

Let $X$ and $Y$ be two manifolds, $W$ be an open subset of $X \times Y$, and $\Omega$ be a closed 2-form on $W$ of type (1,1) with respect to the product structure on $W$. Denote by $X_k$ and $Y_l$ the diagonals of $X^k$ and $Y^l$, respectively, by $\pi_{ij}: X^k \times Y^l \to X \times Y$ the projection that maps $(x_1, \ldots, x_k, y_1, \ldots, y_l) \in X^k \times Y^l$ to $(x_i,y_j)$, and by $i: X \times Y \mapsto X^k \times Y^l$ the diagonal inclusion that maps $(x,y) \in X \times Y$ to $(x, \ldots,x,y,\ldots,y)\in X_k \times Y_l$. The set
\[
\hat W := \bigcap\limits_{i,j} \pi_{ij}^{-1} (W)
\]
is an open subset of $X^k \times Y^l$ containing $i(W)$. Let $C_{ij}$ be a $k \times l$-matrix such that $\sum_i C_{ij}=0$ for all $j$ and $\sum_j C_{ij} =0$ for all $i$. The form
\[
\hat \Omega:= \sum\limits_{i,j} C_{ij}\pi_{ij}^\ast \Omega
\]
is a closed 2-form on $\hat W$ of type (1,1) with respect to the product structure on $X^k \times Y^l$. Let $d=d_x+d_y$ be the corresponding decomposition of the exterior derivative operator on $\hat W$. 
\begin{definition}
Let $W_0$ be a neighborhood of $i(W)$ in $\hat W$.
We say that a function $C$ on $W_0$ is a Calabi function associated with the form $\Omega$ and the matrix $C_{ij}$ if $\hat\Omega = d_xd_y C$, $C|_{(X_k \times Y^l) \cap W_0}=0$, and $C|_{(X^k \times Y_l) \cap W_0}=0$.
\end{definition}

The condition $\hat\Omega = d_xd_y C$ means that $C$ is a potential of the closed (1,1)-form~$\hat\Omega$. We will show that the Calabi functions associated with the form $\Omega$ and the matrix $C_{ij}$ exist. We fix Riemannian metrics on $X$ and $Y$ and for every point $(a,b)\in W$ choose geodesically convex neighborhoods $U_{a,b}$ of $a$ in $X$ and $V_{a,b}$ of $b$ in $Y$ such that $U_{a,b} \times V_{a,b} \subset W$. There exists a potential $\Phi_{a,b}(x,y)$ of the form $\Omega$ on $U_{a,b} \times V_{a,b}$. It is determined up to a summand $f(x) + g(y)$. We define a Calabi function $C_{a,b}$ on $U^k_{a,b} \times V^l_{a,b}$ by the formula
\[
C_{a,b}(x_1, \ldots,x_k,y_1, \ldots,y_l):= \sum\limits_{ij} C_{ij} \Phi_{a,b}(x_i,y_j).
\]
It does not depend on the choice of the potential $\Phi_{a,b}$. 
Moreover, for another point $(c,d) \in W$ the functions $C_{a,b}$ and $C_{c,d}$ agree on
\[
(U^k_{a,b} \times V^l_{a,b}) \cap (U^k_{c,d} \times V^l_{c,d}) = (U_{a,b} \cap U_{c,d})^k  \times (V_{a,b} \cap V_{c,d})^l.
\]
Therefore, the functions $C_{a,b}$ determine a Calabi function $C$ on
\[
W_0 := \bigcup\limits_{(a,b)\in W} U^k_{a,b} \times V^l_{a,b}.
\]
It is easy to check that the Calabi functions associated with the form $\Omega$ and the matrix $C_{ij}$ define the same germ around $i(W)$.

In the rest of the paper we will use a modification of these Calabi functions. Assume that $X=Y$ and $W$ is an open subset of $X^2$ containing the diagonal $X_2$.
Let $C_n$ be a Calabi function associated with a closed (1,1)-form $\Omega$ on $W$ and an $n \times n$-matrix $C_{ij}$. We introduce a modified Calabi function
\[
\mathcal{C}_n(x_1, \ldots, x_n) := C_n(x_1, \ldots, x_n,x_1, \ldots, x_n)
\]
which is defined on a neighborhood of the diagonal $X_n$ in $X^n$. 

\section{A generating function of the Lagrangian manifold of $n$-cycles}\label{S:gamma}

Let $\Omega$ be the closed (1,1)-form on $V_\Gamma \subset M^2$ introduced in Section~\ref{S:gamma} and $C_{ij}$ be an $n \times n$-matrix such that
\[
C_{ij}=
\begin{cases}
1 & \mbox{ if } j=i+1 \mbox{ for } 1 \leq i \leq n-1 \mbox{ or } (i,j)=(n,1);\\
-1 & \mbox{ if } i=j;\\
0 & \mbox{ otherwise.}
\end{cases}
\]
Then there exists a neighborhood $V_0$ of the diagonal $M_n$ in $M^n$ and a Calabi function $\mathcal{C}_n$ on $V_0$ associated with the form $\Omega$ and the matrix~$C_{ij}$.
For any point $\x \in M$ there exists a contractible neighborhood $U$ of $\x$ in $M$ such that $U^2 \subset V_\Gamma$, $U^n \subset V_0$, and\begin{eqnarray}\label{E:ncalabi}
   \mathcal{C}_n(x_1, \ldots,x_n) = \Phi(x_1,x_2) -\Phi(x_2,x_2) + \Phi(x_2, x_3) -\\
     \Phi(x_3,x_3) + \ldots   - \Phi(x_n,x_n) + \Phi(x_n,x_1)  - \Phi(x_1,x_1)  \nonumber
\end{eqnarray}
 for any potential $\Phi(x,y)$ of the form $\Omega$ on $U^2$. We call the function $\mathcal{C}_n$ the cyclic $n$-point Calabi function  of the form $\Omega$. Now we will show that the germ of the generating function of the Lagrangian manifold of $n$-cycles of the local symplectic groupoid constructed in Theorem \ref{T:locgr} around the diagonal $M_n$ of $M^n$ is given by the function~$\mathcal{C}_n$.

If $F(x_1, \ldots,x_n)$ is a function of $n$ variables and $(U, \{x^i\})$ is a coordinate chart on $M$,
we will use the following notations for the partial derivatives of the function $F$ on $U^n$:
\[
\frac{\p F}{\p x_k^i} = \p^k_iF.
\]
In particular, $\Omega = (\p^1_i \p^2_j \Phi)(x,y) dx^i \wedge dy^j$ and
\[
\frac{\p}{\p x^i} (\Phi(x,x)) = (\p^1_i \Phi)(x,x) + (\p^2_i \Phi)(x,x).
\]
Since $\epsilon^\ast \Omega = \omega$, we see that locally
\begin{equation}\label{E:omphi}
\omega = (\p^1_i \p^2_j \Phi)(x,x) dx^i \wedge dx^j.
\end{equation}
Shrinking $U$ around $\x$, if necessary, 
we can find an open ball $B$ in $\R^m$ centered at the origin and functions $g,\tilde g: U^2 \to B$ such that
\begin{equation}\label{E:glamb}
\gamma(x,y) = \lambda(x,g(x,y)) = \tilde\lambda(y, \tilde g(x,y))
\end{equation}
for all $(x,y) \in U^2$ as in the proof of Lemma \ref{L:gamma}. 

We will use the coordinates $(x^i,u^i, y^i)$ on $U^3$ and the dual coordinates $(u^i,\xi_i)$ on $T^\ast U$. Let $\kappa:U^3 \to \hat E$ be the mapping given by
\[
\kappa: (x,u, y) \mapsto (u, g(u,y), \tilde g(x,u)).
\]
Denote by $\pi_i, i=1,2,3$, the projections of $U^3$ onto the respective factors and set $\pi_{12}(x,u,y):=(x,u)$ and $\pi_{23}(x,u,y):=(u,y)$.
Recall that the mappings $\tau:\hat E \to E$ and $\tilde\tau:\hat E \to \tilde E$ are given by $\tau(z,\tilde z)=z$ and $\tilde\tau(z,\tilde z)=\tilde z$. We write the point $(u, g(u,y), \tilde g(x,u))$ of $\hat E$ as $(z,\tilde z)$ with $z=(u, g(u,y))$ and $\tilde z = (u, \tilde g(x,u))$.  We see from (\ref{E:glamb}) that
\begin{equation}\label{E:circ}
\lambda  \tau \kappa = \gamma \pi_{23} \mbox{ and } \tilde\lambda \tilde\tau \kappa = \gamma \pi_{12}. 
\end{equation}
In the calculations below we will use the commutative diagrams
\begin{equation}\label{E:diagr}
\xymatrix{
U^3 \ar@{>}[r]^{\kappa}  \ar@{>}[rd]_-{\gamma \pi_{23}} & W \ar@{>}[r]^{\chi |_{W}} \ar@{>}[d]_-{\lambda \tau}&  V  \ar@{>}[ld]^-s\\
& M & 
}
\quad
\xymatrix{
U^3 \ar@{>}[r]^{\kappa}  \ar@{>}[rd]_-{\gamma  \pi_{12}} & W \ar@{>}[r]^{\chi |_{W}} \ar@{>}[d]_-{\tilde\lambda \tilde\tau}&  V  \ar@{>}[ld]^-t\\
& M & 
}
 \end{equation}
 whose commutativity follows from Theorem \ref{T:locgr}, (\ref{E:sandt}), and (\ref{E:circ}). Pulling back Eqn (\ref{E:dvarphi}) by $\tau \kappa$, we get that
\begin{eqnarray}\label{E:omvc1}
\pi_{23}^\ast \gamma^\ast\omega - \pi_2^\ast \omega = d \kappa^\ast \tau^\ast \phi
\end{eqnarray}
on $U^2$. A simple calculation shows that the equality
\begin{eqnarray}\label{E:omvc}
(\p^1_i \p^2_j\Phi)(u,y)du^i \wedge dy^j - (\p^1_i \p^2_j\Phi)(u,u)du^i \wedge du^j =\\
 d (((\p^1_i \Phi)(u,u) - (\p^1_i \Phi)(u,y)) du^i) \nonumber
\end{eqnarray}
holds on $U^3$. The left-hand sides of (\ref{E:omvc1}) and (\ref{E:omvc}) agree. The 1-forms $\kappa^\ast \tau^\ast \phi$ and $((\p^1_i \Phi)(u,u) - (\p^1_i \Phi)(u,y)) du^i$ on $U^3$ are horizontal with respect to the projection $\pi_2$, vanish on the diagonal of $U^3$, and have equal differentials. Therefore, they coincide:
\begin{equation}\label{E:kappaxy}
\kappa^\ast \tau^\ast \phi = ((\p^1_i \Phi)(u,u) - (\p^1_i \Phi)(u,y)) du^i.
\end{equation}
Similarly, pulling back Eqn (\ref{E:dtvarphi}) by $\tilde\tau\kappa$, we get that
\begin{eqnarray}\label{E:omvc2}
\pi_{12}^\ast \gamma^\ast\omega - \pi_2^\ast \omega = d \kappa^\ast\tilde\tau^\ast\tilde\phi
\end{eqnarray}
on $U^3$. We see from Eqn (\ref{E:omvc2}) and the equality
\begin{eqnarray*}
(\p^1_i \p^2_j\Phi)(x,u)dx^i \wedge du^j - (\p^1_i \p^2_j\Phi)(u,u)du^i \wedge du^j =\\
 d (((\p^2_i \Phi)(x,u) - (\p^2_i \Phi)(u,u)) du^i) \nonumber
\end{eqnarray*}
that
\begin{equation}\label{E:kappaux}
\kappa^\ast\tilde\tau^\ast\tilde\phi = ((\p^2_i \Phi)(x,u) - (\p^2_i \Phi)(u,u)) du^i
 \end{equation}
on $U^3$. We introduce the Calabi function
\begin{equation}\label{E:tcal}
T(x,u,y) = \Phi(x,u) - \Phi(u,u) + \Phi(u,y) - \Phi(x,y)
\end{equation}
on $U^3$ and obtain from (\ref{E:kappaxy}) and (\ref{E:kappaux}) that
\begin{eqnarray*}
\frac{\p T(x,u,y)}{\p u^i}du^i  = ((\p^2_i\Phi)(x,u) - (\p^2_i\Phi)(u,u))du^i \hskip 1cm\\
 - ((\p^1_i \Phi)(u,u) - (\p^1_i\Phi)(u,y))du^i = \kappa^\ast (\tilde\tau^\ast\tilde\phi - \tau^\ast \phi).\nonumber
\end{eqnarray*}
It follows that the mapping $\chi \kappa : U^3 \to V$ maps $(x,u,y)$ to
\[
(u^i,\xi_i) = \left(u^i, \frac{\p T(x,u,y)}{\p u^i}\right).
\]
We see from the diagrams (\ref{E:diagr}) that
\begin{equation}\label{E:gsgt}
\gamma(u,y)=s\left(u, \frac{\p T(x,u,y)}{\p u}\right) \mbox{ and } \gamma(x, u)=t\left(u, \frac{\p T(x,u,y)}{\p u}\right).
 \end{equation}
 One can check that the Calabi functions $\mathcal{C}_n$ and $T$ satisfy the identity
 \[
 \frac{\p}{\p x_k^i} T(x_{k-1},x_k,x_{k+1}) = \frac{\p \mathcal{C}_n}{\p x_k^i}
 \]
for $k=1,\ldots,n$, where we identify $x_0$ with $x_n$ and $x_{n+1}$ with $x_1$.  Now (\ref{E:gsgt}) implies that
\[
s\left(x_{k-1}, \frac{\p \mathcal{C}_n}{\p x_{k-1}}\right) = \gamma(x_{k-1},x_k) = t\left(x_k, \frac{\p \mathcal{C}_n}{\p x_k}\right),
\]
which means that
\[
\left\{\left(x_k, \frac{\p \mathcal{C}_n}{\p x_k}\right), k = 1, \ldots, n\right\}
\]
is an $n$-cycle in $T^\ast U \cap V$. We have thus proved the following theorem.
\begin{theorem}
The Lagrangian manifold of $n$-cycles of the local symplectic groupoid constructed in Theorem \ref{T:locgr} has a generating function whose germ around the diagonal $M_n$ of $M^n$ is given by the cyclic $n$-point Calabi function~$\mathcal{C}_n$.
\end{theorem}

\section{A critical point of the function $u\mapsto T(x,u,y)$}

In Section \ref{S:gamma} we constructed the mapping $\gamma: V_\Gamma \to M$, where $V_\Gamma$ is a neighborhood of the diagonal $M_2$ in $M^2$. It satisfies (\ref{E:defg}), (\ref{E:maing}), and (\ref{E:geps}). We also defined the Calabi function $T=T(x,u,y)$ in terms of a local potential of the form $\Omega$. The function $T$ does not depend on the choice of local potentials and is defined on some neighborhood $W_T$ of the diagonal $M_3$ in $M^3$.

In this section we will show that if $x$ and $y$ are close, then $\gamma(x,y)$ is a nondegenerate critical point of the function $u \mapsto T(x,u,y)$ with zero critical value, $T(x, \gamma(x,y),y)=0$. Heuristically it means that for some constant $C$ the oscillatory integral operator
\[
f(x) \mapsto C \hbar^{-m} \int e^{\frac{i}{\hbar}(\Phi(x,y) - \Phi(y,y))} f(y)\, dy
\]
is an approximate projector.

Let $(x,y) \mapsto (z(x,y), \tilde z(x,y))$ be the mapping from $V_\Gamma$ to $W_\Gamma$ inverse to $\pi\times\tilde\pi|_{W_\Gamma}: W_\Gamma \to V_\Gamma$. Thus, $\pi(z(x,y))=x$ and $\tilde\pi(\tilde z(x,y))=y$. We have $z(x,x)=z_x$ and $\tilde z(x,x)=\tilde z_x$.

\begin{lemma}\label{L:uprpr}
\begin{enumerate}[(i)]
\item There is a neighborhood $V'$ of $M_2$ in $V_\Gamma$ such that
\begin{equation}\label{E:gggg}
(\gamma(x,y),y) \in V_\Gamma\mbox{ and }\gamma(\gamma(x,y),y)=\gamma(x,y)
 \end{equation}
for any $(x,y)\in V'$.
\item There is a neighborhood $V''$ of $M_2$ in $V_\Gamma$ such that
\begin{equation}\label{E:gggg}
 (x, \gamma(x,y)) \in V_\Gamma\mbox{ and }\gamma(x, \gamma(x,y))=\gamma(x,y)
 \end{equation}
for any $(x,y)\in V''$.
\end{enumerate}
\end{lemma}
\begin{proof}
Suppose that $(x,y) \in V_\Gamma$. Then $\gamma(x,y)=\lambda(z(x,y)) = \tilde\lambda(\tilde z(x,y))$. It follows from the equalities $\lambda(z_{\gamma(x,y)}) = \gamma(x,y)=\tilde\lambda(\tilde z(x,y))$ that
the range of the mapping
\begin{equation}\label{E:uprime}
 V_\Gamma \ni (x,y) \mapsto (z_{\gamma(x,y)}, \tilde z(x,y))
 \end{equation}
lies in $\Gamma$. The mapping (\ref{E:uprime}) maps $(x,x)$ to $(z_x, \tilde z_x)$ for all $x \in M$ and thus it maps $M_2$ to $Z_{\hat E}$.
Denote by $V'$ the inverse image of $W_\Gamma$ with respect to the mapping (\ref{E:uprime}). It is easy to check that $V'$ satisfies the conditions of part (i) of the lemma. Part (ii) can be proved similarly.
\end{proof}
We define mappings $\alpha,\beta: V_\Gamma \to M^2$ as follows:
\[
\alpha: (x,y) \mapsto (\gamma(x,y),y) \mbox{ and } \beta: (x,y) \mapsto (x, \gamma(x,y)).
\]
Lemma \ref{L:uprpr} implies that
\begin{eqnarray}\label{E:alal}
\alpha(V') \subset V_\Gamma \mbox{ and } \alpha \circ (\alpha|_{V'}) = \alpha|_{V'}; \\
 \beta(V'')\subset V_\Gamma \mbox{ and } \beta \circ (\beta|_{V''}) = \beta|_{V''}. \nonumber
 \end{eqnarray}
If $(x,y)\in V'$, then $\beta\alpha(x,y)=(\gamma(x,y),\gamma(x,y))$ and if $(x,y)\in V''$, then $\alpha\beta(x,y)=(\gamma(x,y),\gamma(x,y))$. Observe that $\epsilon\gamma(x,y)=(\gamma(x,y),\gamma(x,y))$ for $(x,y)\in V_\Gamma$. Hence,
\begin{equation}\label{E:albeta}
\alpha \circ (\beta|_{V''}) = \epsilon\gamma|_{V''} \mbox{ and } \beta \circ (\alpha|_{V'}) = \epsilon\gamma|_{V'}.
 \end{equation}

Let $\x$ be an arbitrary point in $M$ and $U$ be a contractible neighborhood of $\x$ in $M$. One can shrink $U$ around $\x$ so that if $\alpha^\ast$ is applied to (\ref{E:omom}), it will follow from (\ref{E:alal}) and (\ref{E:albeta}) that
\[
\alpha^\ast\Omega = \alpha^\ast \beta^\ast\Omega=\Omega
\]
on $U^2$. Similarly, we can assume that $\beta^\ast \Omega = \Omega$ on $U^2$. Let $\Phi$ be a potential of $\Omega$ on $U$, so that
\[
   \Omega = d_x d_y \Phi = d d_y \Phi,
\]
where $d= d_x+d_y$ is the decomposition of the exterior derivative with respect to the product structure on $M^2$. Shrinking $U$ further, we can assume that $\Phi(\gamma(x,y),y), \Phi(x,\gamma(x,y))$, and $\Phi(\gamma(x,y),\gamma(x,y))$ are defined for all $x,y\in U$. We obtain from $\alpha^\ast\Omega=\Omega$ that
\[
d(\alpha^\ast d_y \Phi - d_y\Phi)=0.
\]
It is easy to check that the closed form $\alpha^\ast d_y \Phi - d_y\Phi$ is of type (0,1) with respect to the product structure on $U^2$ and vanishes on the diagonal of~$U^2$. Therefore, it vanishes identically on $U^2$. Similarly, the form $\beta^\ast d_x \Phi - d_x\Phi$ also vanishes on $U^2$. Thus,
\begin{equation}\label{E:vanish}
     \alpha^\ast d_y \Phi = d_y\Phi \mbox{ and } \beta^\ast d_x\Phi = d_x\Phi
 \end{equation}
on $U^2$. It follows from (\ref{E:albeta}) that 
\[
T(x,\gamma(x,y),y) = \Phi - \alpha^\ast \Phi - \beta^\ast \Phi + \alpha^\ast\beta^\ast\Phi
\]
for all $x,y\in U$. Differentiating this equality, we get from (\ref{E:vanish}) that
\begin{eqnarray*}
d(T(x,\gamma(x,y),y)) = d_x\Phi - \alpha^\ast d_x\Phi - \beta^\ast d_x\Phi + \alpha^\ast\beta^\ast d_x \Phi + \\
d_y\Phi - \alpha^\ast d_y\Phi - \beta^\ast d_y\Phi + \alpha^\ast\beta^\ast d_y \Phi =\\
 - \alpha^\ast d_x\Phi + \alpha^\ast\beta^\ast d_x \Phi  - \beta^\ast d_y\Phi + \beta^\ast \alpha^\ast d_y \Phi =\\
 -\alpha^\ast(d_x\Phi - \beta^\ast \Phi) - \beta^\ast(d_y\Phi - \alpha^\ast\Phi)=0.
\end{eqnarray*}
Since the function $T(x,\gamma(x,y),y)$ vanishes on the diagonal of $U^2$ and is constant, it vanishes identically on $U^2$.
Since $U^2$ is a neighborhood of the diagonal point $(\x,\x)\in M_2$ in $M^2$ and $\x\in M$ is arbitrary, we arrive at the following statement.
\begin{lemma}\label{L:crval}
There exists a neighborhood $V_0$ of $M_2$ in $V_\Gamma$ such that $(x, \gamma(x,y),y)\in W_T$ and the
 identity $T(x,\gamma(x,y),y)=0$ holds for any $(x,y) \in V_0$.
\end{lemma}
\begin{lemma}\label{L:nondeg}
There exists a neighborhood $V_1$ of $M_2$ in $V_\Gamma$ such that $(x, \gamma(x,y),y)\in W_T$
and the point $\gamma(x,y)$ is a nondegenerate critical point of the function $u \mapsto T(x,u,y)$ for any $(x,y)\in V_1$. 
\end{lemma}
\begin{proof}

Assume that $U$ is a neighborhood of an arbitrary point $\x$ in $M$ such that $U^2 \subset V' \cap V''$, where $V'$ and $V''$ are as in Lemma~\ref{L:uprpr}. We additionally assume that $U$ lies in a contractible coordinate chart $(\hat U, \{x^i\})$, is such that $\gamma(U^2) \subset \hat U$, and $\Phi$ is a potential of the form $\Omega$ on $\hat U^2$. It will be convenient to denote $\gamma(x,y)$ by $\bar \gamma$. It follows from (\ref{E:vanish}) that in coordinates
\[
(\p^1_i\Phi)(x,\bar\gamma)= (\p^1_i\Phi)(x,y) \mbox{ and } (\p^2_i\Phi)(\bar\gamma,y) = (\p^2_i\Phi)(x,y).
\]
Applying $\alpha^\ast$ to the first identity and $\beta^\ast$ to the second one, we get from Lemma \ref{L:uprpr} that
\[
(\p^1_i\Phi)(\bar\gamma,\bar\gamma)= (\p^1_i\Phi)(\bar\gamma,y) \mbox{ and } (\p^2_i\Phi)(\bar\gamma,\bar\gamma) = (\p^2_i\Phi)(x,\bar\gamma).
\]
Using these identities, we obtain that
\begin{eqnarray*}
\frac{\p}{\p u^i}  T(x,u,y)|_{u=\bar\gamma} = (\p^2_i\Phi)(x, \bar\gamma) - (\p^1_i\Phi)(\bar\gamma,\bar\gamma) - \\
(\p^2_i\Phi)(\bar\gamma,\bar\gamma) + (\p^1_i\Phi)(\bar\gamma,y) =0.
\end{eqnarray*}
Thus, $\bar\gamma=\gamma(x,y)$ is a critical point of the function $u \mapsto T(x,u,y)$ for all $x,y\in U$. Now we will show that there exists a neighborhood $U_1$ of $\x$ in $U$ such that this critical point is nondegenerate for all $x,y \in U_1$. Writing the form $\Omega$ on $U$ in local coordinates as
\[
\Omega = \frac{1}{2} \omega_{kl}(\gamma) d\gamma^k \wedge d\gamma^l,
\]
we obtain from Lemma \ref{L:oneone} that
\begin{equation}\label{E:omcoord}
\omega_{kl}(\gamma) \frac{\p \gamma^k}{\p x^i} \frac{\p \gamma^l}{\p x^j}=0  \mbox{ and }  \omega_{kl}(\gamma) \frac{\p \gamma^k}{\p y^i} \frac{\p \gamma^l}{\p y^j}=0.
\end{equation}
Consider the matrices
\[
A := ((\p^1_i \gamma^k)(x,x)), B := ((\p^2_i \gamma^k)(x,x)),  \mbox{ and } C := (\omega_{kl}(x)),
\]
where $x\in U$. Since $\gamma(x,x)=x$, we get that $A + B = 1$  (the identity matrix). We derive from (\ref{E:omcoord}) that
\[
A^t C A =0 \mbox{ and }  B^t C B = 0,
\]
whence it follows that
\[
C = A^t C + CA.
\]
Therefore, $B=C^{-1} A^t C$ and $BA=C^{-1} (A^t C A)=0$, which implies that $A$ and $B$ are complementary projectors. 

{\it Remark} \ It can be seen from (\ref{E:omcoord}) that for a fixed $x$ both projectors $A$ and $B$ are of rank $m$.

We have that
\[
\Omega = \omega_{kl}(\gamma) \frac{\p \gamma^k}{\p x^i} \frac{\p\gamma^l}{\p y^j}dx^i \wedge dy^j = (\p^1_i \p^2_j\Phi)(x,y) dx^i \wedge dy^j
\]
on $U$. Therefore,
\[
\omega_{kl}(\gamma) \frac{\p \gamma^k}{\p x^i} \frac{\p\gamma^l}{\p y^j}= (\p^1_i \p^2_j\Phi)(x,y).
\]
Setting $y=x$, we get
\begin{equation}\label{E:atcb}
A^t C B = ((\p^1_i \p^2_j\Phi)(x,x)),
\end{equation} 
whence
\begin{equation}\label{E:btca}
B^t C A = -((\p^1_j \p^2_i\Phi)(x,x)).
\end{equation} 
We want to show that if $x$ and $y$ are close, then the critical point $u=\gamma(x,y)$ of the function $u \mapsto T(x,u,y)$ is nondegenerate, that is, the Hessian matrix
\[
((\p^2_i\p^2_j T)(x, \gamma(x,y),y))
\]
is nondegenerate. Here $\p^2_i$ is the partial derivative with respect to $u^i$. It suffices to show that it holds true if $x=y$, that is, the matrix
\[
((\p^2_i\p^2_j T)(x,x,x))
\]
is nondegenerate. We have
\begin{eqnarray*}
\frac{\p^2 T(x,u,y)}{\p u^i \p u^j} = (\p^2_i \p^2_j \Phi)(x, u) - (\p^1_i\p^1_j\Phi) (u,u)  -  (\p^1_i \p^2_j \Phi)(u,u) -\\ 
(\p^2_i \p^1_j \Phi)(u,u) - (\p^2_i \p^2_j \Phi)(u,u) + (\p^1_i \p^1_j \Phi)(u,y),
\end{eqnarray*}
whence 
\begin{eqnarray*}
(\p^2_i\p^2_j T)(x,x,x) = (\p^2_i \p^2_j \Phi)(x, x) - (\p^1_i\p^1_j\Phi) (x,x)  -  (\p^1_i \p^2_j \Phi)(x,x) -\\ 
(\p^2_i \p^1_j \Phi)(x,x) - (\p^2_i \p^2_j \Phi)(x,x) + (\p^1_i \p^1_j \Phi)(x,x) =\\
 -  (\p^1_i \p^2_j \Phi)(x,x) - (\p^1_j \p^2_i \Phi)(x,x).
\end{eqnarray*}

Using that $A+B=1$, we get from (\ref{E:atcb}) and (\ref{E:btca}) that
\begin{eqnarray*}
((\p^2_i\p^2_j T)(x,x,x)) = - A^t CB + B^t CA =\\
 - A^t C + A^t CA + CA - A^t CA = - A^t C + CA=\\
 C(A - C^{-1} A^t C) = C(A-B).
\end{eqnarray*}
Since $A$ and $B$ are complementary projectors, $A-B$ is invertible. Therefore, the Hessian matrix $((\p^2_i\p^2_j T)(x,x,x))$ is nondegenerate.

\end{proof}

Lemmas \ref{L:crval} and \ref{L:nondeg} imply the following theorem.
\begin{theorem}
There exists a neighborhood $V_T$ of $M_2$ in $V_\Gamma$ such that $(x,\gamma(x,y),y) \in W_T$ and $\gamma(x,y)$ is a nondegenerate critical point of the function $u \mapsto T(x,u,y)$ with zero critical value for all $(x,y) \in V_T$.
\end{theorem}
\begin{center}
{\bf Discussion} 
\end{center}

The geometric objects considered in this paper are germs of real smooth submanifolds of a vector bundle containing its zero section or of $M^n$ containing the diagonal $M_n$. One can generalize the results of this paper in two possible ways. Instead of fields of real Lagrangian submanifolds, one can consider fields of H\"ormander's complex Lagrangian ideals. We expect that then one can use germs of complex almost analytic manifolds introduced in \cite{MS} in order to construct a local symplectic groupoid over an almost-K\"ahler manifold from two transversal fields of Lagrangian ideals. For this groupoid, the form $\Omega$, its local potentials, and the corresponding Calabi functions will be complex-valued. We expect that this groupoid should correspond to Berezin quantization via covariant symbols on general symplectic manifolds (see \cite{LC} and references therein). One can also consider formal symplectic groupoids introduced in \cite{CDF} and \cite{CMP2005}, where germs of submanifolds are replaced with their $\infty$-jets. Such formal symplectic groupoids over almost-K\"ahler manifolds were constructed in~\cite{LMP2003}.

\end{document}